\documentclass{amsart}

\newtheorem{theorem}{Theorem}[section]

\newtheorem{proposition}[theorem]{Proposition}
\newtheorem{corollary}[theorem]{Corollary}

\theoremstyle{definition}
\newtheorem{definition}[theorem]{Definition}

\theoremstyle{remark}
\newtheorem{remark}[theorem]{Remark}

\numberwithin{equation}{section}

\newcommand{\ra}{\rightarrow}

\newcommand{\C }{ \mathbb{C}}

\newcommand{\Z}{\mathbb{Z}}
\newcommand{\Q}{\mathbb{Q}}

\newcommand{\Dh}{\mathcal{D}} 
\newcommand{\subsetneq}{\subset_{_{_{{\!\!\!\!\!\not{=}}}}}}
\newcommand{\supsetneq}{\supset_{_{_{{\!\!\!\!\!\not{=}}}}}}

\begin{document}

\title{The dihedral group $\Dh_5$ as group of symplectic automorphisms on K3 surfaces}
\author{Alice Garbagnati}

\address{Dipartimento di Matematica, Universit\`a di Milano,
  via Saldini 50, I-20133 Milano, Italia}
\email{alice.garbagnati@unimi.it}
\thanks{I would like to thank Bert van Geemen, who helps and
encourages me during the preparation of this paper and Alessandra
Sarti for several interesting discussions. I also thank the referee for very useful comments}

\subjclass[2000]{Primary 14J28, 14J50}

\keywords{K3 surfaces, symplectic automorphisms, dihedral groups, lattices.}

\begin{abstract}
We prove that if a K3 surface $X$ admits $\Z/5\Z$ as group of
symplectic automorphisms, then it actually admits $\Dh_5$ as group
of symplectic automorphisms. The orthogonal complement to the
$\Dh_5$-invariants in the second cohomology group of $X$ is a rank
16 lattice, $L$. It is known that $L$ does not depend on $X$: we
prove that it is isometric to a lattice recently described by R.
L. Griess Jr. and C. H. Lam. We also give an elementary
construction of $L$.
\end{abstract}

\maketitle
\markboth{ALICE GARBAGNATI}{THE DIHEDRAL GROUP $\Dh_5$ AS GROUP OF SYMPLECTIC AUTOMORPHISMS ON K3}

\section{Introduction}

A finite group of symplectic automorphisms on a K3 surface $X$ has
the property that the desingularization of the quotient of $X$ by
this group is again a K3 surface. In \cite{Nikulin symplectic} the
finite abelian groups of symplectic automorphisms on a K3 surface
are classified. The main result of Nikulin in \cite{Nikulin
symplectic} is that the isometries induced by finite abelian
groups of symplectic automorphisms on the second cohomology group
of a K3 surface are essentially unique. The uniqueness of the
isometries induced by $G$ on $H^2(X,\Z)$ implies that the lattice
$\Omega_G:=(H^2(X,\Z)^G)^{\perp}$ does not depend on $X$. Thanks
to this result it is possible to associate  the lattice
$\Omega_G$ to each finite abelian group $G$ of symplectic automorphisms on a K3 surface. From this one obtains information on the coarse moduli
space of K3 surfaces admitting $G$ as group of symplectic
automorphisms (cf. \cite{Nikulin symplectic}, \cite{symplectic
prime}, \cite{symplectic not prime}). In \cite{symplectic prime}
and \cite{symplectic not prime} the lattices $\Omega_G$ are
computed for each finite
abelian group $G$ of symplectic automorphisms on a K3 surface.\\
In \cite{mukai} and \cite{Xiao} the finite (not necessary abelian)
groups of symplectic automorphisms on a K3 surface are classified. Under some conditions (cf. Section \ref{section: symplectic automorphisms on K3 surfaces}) Nikulin's result, on the uniqueness of the isometries induced by
finite groups of symplectic automorphisms on the second cohomology
group of the K3 surfaces, can be extended to finite (not
necessary abelian) groups (cf. \cite{whitcher}). As a
consequence one can attach the lattice
$\Omega_G:=(H^2(X,\Z)^G)^{\perp}$, which depends only on $G$, also to some finite non abelian groups $G$
of symplectic automorphisms on a K3 surface $X$.\\

Let us consider a pair of finite groups $(G,H)$ such that $G$ acts symplectically on a K3 surface and $H$ is a subgroup of $G$. It is evident that the K3 surfaces admitting $G$ as group of symplectic automorphisms, admits also $H$ as group of symplectic automorphisms. It is more surprising that for certain pairs of groups $(G,H)$ also the viceversa holds indeed for certain pair $(G,H)$ the condition ``a K3 surface  $X$ admits $G$ as group of symplectic automorphism" is equivalent to the condition ``$X$ admits $H$ as group of symplectic automorphisms". For these pairs the lattices $\Omega_G$ and $\Omega_H$ coincide. The aim of this paper is to describe this situation and to give explicitely one pair $(G,H)$. We observe that in order to find $(G,H)$ with the described property, one has to consider non abelian groups acting symplectically on K3 surfaces, indeed the lattices $\Omega_K$ associated to abelian groups $K$ are completely described in \cite{symplectic prime}, \cite{symplectic not prime} and one can check that $\Omega_G$ and $\Omega_H$ never coincide if both $G$ and $H$ are abelian and $G\neq H$.\\

In the {\it Section} \ref{section: symplectic automorphisms on K3 surfaces} we describe some known results on symplectic automorphisms over k\"ahlerian K3 surfaces and we prove our main results (Proposition \ref{prop: G as symp group iff H as symp group} and Corollary \ref{cor: Z/5 iff Dh5}).
In the Proposition \ref{prop: G as symp group iff H as symp group} we give sufficient conditions on $G$ and $H$ to prove that a K3 surface admits $G$ as group of symplectic automorphisms if and only if it admits $H$ as group of symplectic automorphisms.  Applying this proposition we prove (Corollary \ref{cor: Z/5 iff Dh5}) that a K3 surface admits $\Z/5\Z$ as group of symplectic automorphisms if and only if it admits $\Dh_5$ (the dihedral group of order 10) as group of symplectic automorphisms. In particular we prove that $\Omega_{\Z/5\Z}\simeq \Omega_{\Dh_5}$\\
In \cite{symplectic prime} the isometry induced on $\Omega_{\Z/5Z}$ by a symplectic automorphim of order 5, is described. Since we prove that $\Omega_{\Z/5\Z}\simeq\Omega_{\Dh_5}$, there is also an involution acting on this lattice. In order to describe both the isometry of order 5 and the involution generating $\Dh_5$ on $\Omega_{\Dh_5}$, in the {\it Section} \ref{section: construction of omega5} we give a different description of this lattice: it is an overlattice of $A_4(-2)^{\oplus 4}$ (the isometry of order 5 is induced by the natural one on $A_4$). In the proof of the Corollary \ref{cor: omegadh5 is isomteric to DIH10(16)} also the action of the involution is described. Moreover we show that the lattice $\Omega_{\Z/5\Z}$ (computed in \cite{symplectic prime}) is isometric to a lattice describe by Griess and Lam in \cite{EE8- lattices and dihedral}.\\
In the {\it Section} \ref{section: examples} we consider algebraic K3 surfaces and in particular 3-dimensional families of K3 surfaces admitting a symplectic automorphisms of order 5, $\sigma_5$, and a polarization, invariant under $\sigma_5$. It follows from the results of the Section \ref{section: symplectic automorphisms on K3 surfaces} that the K3 surfaces in these families admit also an involution $\iota$, generating together with $\sigma_5$ the dihedral group $\Dh_5$. For each of these families we exhibit the automorphism $\sigma_5$ and we find the automorphism $\iota$.

\section{Symplectic automorphisms on K3
surfaces.}\label{section: symplectic automorphisms on K3 surfaces}
\begin{definition} Let $X$ be a smooth compact complex surface. The surface $X$ is a K3 surface if the canonical bundle of $X$ is trivial and the irregularity of $X$, $q(X):=h^{1,0}(X)$, is 0.
\end{definition}
The second cohomology group of a K3 surface, equipped with the cup product, is isometric to a lattice, which is the unique, up to isometries, even unimodular lattice with signature $(3,19)$. This lattice will be denoted by $\Lambda_{K3}$ and is isometric to $U\oplus U\oplus U\oplus E_8(-1)\oplus E_8(-1)$, where $U$ is the unimodular lattice with bilinear form $\left[\begin{array}{ll}0&1\\1&0\end{array}\right]$ and $E_8(-1)$ is the lattice obtained multipying by $-1$ the lattice associated to the Dynking diagram $E_8$.\\
The N\'eron Severi group of a K3 surface $X$, $NS(X)$, coincide with its Picard group. The transcendental lattice of $X$, $T_X$, is the orthogonal to $NS(X)$ in $H^2(X,\Z)$.  
\begin{definition} An isometry $\alpha$ of $H^2(X,\Z)$ is an effective isometry if it preserves the K\"ahler cone of $X$. An isometry $\alpha$ of $H^2(X,\Z)$ is an Hodge isometry if its $\C$-linear extension to $H^2(X,\C)$ preserves the Hodge decomposition of $H^2(X,\C)$. \end{definition}
\begin{theorem}{\rm (\cite{Burns Rapoport torelli theorem K3})} Let $X$ be a K3 surface and $g$ be an automorphism of $X$, then $g^*$ is an effective Hodge isometry of $H^2(X,\Z)$. Viceversa, let $f$ be an effective Hodge isometry of $H^2(X,\Z)$, then $f$ is induced by an automorphism of $X$.\end{theorem}

\begin{definition} An automorphism $\sigma$ on a K3 surface $X$ is symplectic if $\sigma^*$ acts as the identity on $H^{2,0}(X)$, that is $\sigma^*(\omega_X)=\omega_X$, where $\omega_X$ is a nowhere vanishing holomorphic two form on $X$.\\ Equivalently $\sigma$ is symplectic if the isometry induced by $\sigma^*$ on the transcendental lattice is the identity. \end{definition}

\begin{remark}{\rm Let $\sigma$ be an automorphism of finite order on a K3 surface. The desingularization of $X/\sigma$ is a K3 surface if and only if $\sigma$ is symplectic.}\end{remark}

In \cite{Nikulin symplectic} the finite abelian groups acting symplectically on a k\"ahlerian K3 surface are analyzed. Now it is known that every K3 surface is a K\"ahler variety, \cite{Siu}, so there are no restrictions on the K3 surfaces analyzed in \cite{Nikulin symplectic}.\\
From now on we always assume that $G$ is a finite group of symplectic automorphisms on the K3 surface $X$.
\begin{definition} Let the K3 surface $Y$ be the minimal desingularization of the quotient $X/G$. Let $M_j$ be the curves arising form the desingularization of the singularities of $X/G$.
\end{definition}

The singularities of the quotient $X/G$ are computed by Nikulin (\cite[Sections 6]{Nikulin symplectic}) if $G$ is an abelian group, and by Xiao (\cite[Table 2 ]{Xiao}) for all the other finite groups. If either $G=Q_8$ (binary dihedral group of order 8) or $G=T_{24}$ (binary tetrahedral group of order 24), then there are two possible configurations for the singularities of $X/G$, and hence for the exceptional curve $M_j$ on $Y$. For all the other groups the number and the type of the singularities of $X/G$ are determined by $G$.

\begin{definition} Let us assume that $G\neq Q_8$, $G\neq T_{24}$. The minimal primitive sublattice of $NS(Y)$ containing the curves $M_j$ does not depend on $X$ (cf. \cite{Nikulin symplectic}, \cite{Xiao}). It will be denoted by $M_G$.\end{definition}
The lattice $M_G$ is computed by Nikulin (\cite[Section 7]{Nikulin symplectic}) for each abelian group $G$, and by Xiao (\cite[Table 2]{Xiao}) for the all the other gorups $G$.
\begin{remark}\label{rem: Nikulin lattice} For $G=\Z/2\Z$, the lattice $M_{\Z/2\Z}$ (called Nikulin lattice) is an even overlattice of index 2 of $A_1(-1)^{\oplus 8}$. Its discriminant group is $(\Z/2\Z)^6$ (cf. \cite{Nikulin symplectic}) and its discriminant form is the same as $U(2)^{\oplus 3}$ (cf. \cite{morrison}).\end{remark}
\begin{definition}{\rm (\cite[Definition 4.6]{Nikulin symplectic})} We say that $G$ has a unique action on $\Lambda_{K3}$ if, given two embeddings $i: G\hookrightarrow Aut(X)$, $i': G\hookrightarrow Aut(X')$ such that $G$ is a group of symplectic automorphisms on the K3 surfaces $X$ and $X'$, there exists an isometry $\phi:H^2(X,\Z)\ra H^2(X',\Z)$ such that $i'(g)^*=\phi\circ i(g)\circ \phi^{-1}$ for all $g\in G$.\end{definition}

\begin{theorem}\label{theorem: unique action}{\rm (\cite[Theorem 4.7]{Nikulin symplectic}, \cite[Corollary 3.0.1]{whitcher})} Let $G$ be a finite group acting symplectically on a K3 surface, $G\neq Q_8$, $G\neq T_{24}$. If $M_G$ admits a unique primitive embedding in $\Lambda_{K3}$, then $G$ has a unique action on $\Lambda_{K3}$.\end{theorem}

\begin{definition}
Under the assumptions of the Theorem \ref{theorem: unique action} the lattice $(\Lambda_{K3}^G)^{\perp}$ is uniquely determined by $G$, up to isometry. It will be called $\Omega_G$.
\end{definition}

\begin{remark}\label{rem: rank omega rank m} By \cite[(8,12)]{Nikulin symplectic}, it follows that rank$(\Omega_G)=$rank$(M_G)$.\end{remark}

\begin{theorem}\label{theorem: X admits G iff omegaG in NS(X)}{\rm (\cite[Theorem 4.15]{Nikulin symplectic}) } Let $G$ be a finite group acting symplectically on a K3 surface, such that $G$ has a unique action on $\Lambda_{K3}$. A K3 surface $X$ admits $G$ as group of symplectic automorphisms if and only if the lattice $\Omega_G$ is primitively embedded in $NS(X)$.\end{theorem}

Nikulin proved that for each abelian group acting symplectically on a K3 surface, the hypothesis of the Theorem \ref{theorem: unique action} (and hence the ones of the Theorem \ref{theorem: X admits G iff omegaG in NS(X)}) are satisfied.

\begin{proposition}\label{prop: G as symp group iff H as symp group} Let $G$ be a finite group acting symplectically on a K3 surface and let $H$ be a subgroup of $G$. Let us assume that both $G$ and $H$ are neither $Q_8$ or $T_{24}$. We assume that both $M_H$ and $M_G$ admit a unique primitive embedding in $\Lambda_{K3}$ and rank$(M_G)=$rank$ (M_H)$. Then $\Omega_H\simeq \Omega_G$ and so a K3 surface $X$ admits $G$ as group of symplectic automorphisms if and only if $X$ admits $H$ as group of symplectic automorphisms.
\end{proposition}
\proof Since $H$ is a subgroup of $G$, $\Omega_H$ is a sublattice of $\Omega_G$. Moreover rank$(\Omega_G)=$rank$(\Omega_H)$, by the Remark  \ref{rem: rank omega rank m} and the condition on the rank of the lattices $M_G$ and $M_H$. This implies that $\Omega_H\hookrightarrow \Omega_G$ with a finite index. Let $X$ be a K3 admitting $G$ as group of symplectic automorphisms. Then both $\Omega_G$ and $\Omega_H(\hookrightarrow \Omega_G)$ are primitively embedded in $NS(X)$, hence the index of the inclusion $\Omega_H\hookrightarrow \Omega_G$ is 1, i.e. $\Omega_G\simeq \Omega_H$.\\
The K3 surface $X$ admits $G$ as group of symplectic automorphisms if and only if $\Omega_G$ is primitively embedded in $NS(X)$. By the isometry $\Omega_H\simeq\Omega_G$ this condition is equivalent to require that $\Omega_H$ is primitively embedded in $NS(X)$, which holds  if and only if $X$ admits $H$ as group of symplectic automorphisms.\endproof

\begin{corollary}\label{cor: Z/5 iff Dh5} A K3 surface admits $\Z/5\Z$ as group of symplectic automorphisms if and only if it admits $\Dh_5$ as group of symplectic automorphisms.\end{corollary}
\proof The lattice $M_{\Z/5\Z}$ is computed in \cite{Nikulin symplectic}, where it is proved that it admits a unique primitive embedding in $\Lambda_{K3}$ and that its rank is $16$. The lattice $M_{\Dh_5}$ is described in \cite{Xiao} as an overlattice of index 2 of the lattice $A_4(-1)^{\oplus 2}\oplus A_1(-1)^{\oplus 8}$. In particular rank$(M_{\Dh_5})=16$ and $M_{\Dh_5}\simeq A_4(-1)^{\oplus 2}\oplus M_{\Z/2\Z}$, where $M_{\Z/2\Z}$ is the Nikulin lattice (see Remark \ref{rem: Nikulin lattice}). Thus the discriminant group of $M_{\Dh_5}$ is $(\Z/5\Z)^2\oplus (\Z/2\Z)^6$. By \cite[Theorem 1.14.4]{Nikulin bilinear}, $M_{\Dh_5}$ admits a unique primitive embedding in $\Lambda_{K3}$. The corollary immediately follows from the Proposition \ref{prop: G as symp group iff H as symp group}.\endproof

\begin{remark}{\rm We proved that if a K3 surface admits a symplectic automorphism of order five $\sigma_5$, hence it admits also a symplectic involution generating $\Dh_5$ together with $\sigma_5$. This result can not be improved, i.e.\ it is {\bf not} true that if a K3 surface $X$ admits $\Dh_5=\langle \sigma_5,\iota\rangle$ as group of symplectic automorphisms, then it admits also a symplectic automorphism $\alpha$ such that $J:=\langle \alpha, \sigma_5, \iota\rangle\supsetneq \Dh_5$ is a finite group. By contradiction assuming there exists such an $\alpha$, then $\Dh_5\subsetneq J$ and $\Omega_{\Dh_5}\simeq \Omega_J$. In particular rank$\Omega_J=$rank$\Omega_{\Dh_5}=16$, but there are no finite groups $J$ of symplectic automorphisms on a K3 surface such that $\Dh_5\subsetneq J$ and rank$M_J$(=rank$\Omega_{J})=16$ (cf. \cite[Table 2]{Xiao}). }\end{remark} 

\section{Construction of $\Omega_{\Z/5\Z}\simeq
\Omega_{\Dh_5}$}\label{section: construction of omega5}

The aim of this section is to construct the lattice
$\Omega_{\Z/5\Z}$ as overlattice of $A_4(-2)^{\oplus 4}$ and to describe the action of $\Dh_5$ on this lattice.
The automorphism of order five on $\Omega_{\Z/5\Z}$
will be induced by the automorphism of order five on each copy of
$A_4(-2)$.

We will add some rational linear combinations of the
elements of $A_4(-2)^{\oplus 4}$ to obtain an even overlattice of
$A_4(-2)^{\oplus 4}$. The main point is that we would like to
extend the automorphism of $A_4(-2)^{\oplus 4}$ to the lattice
$\Omega_{\Z/5\Z}$, so if we add an element to $A_4(-2)^{\oplus 4}$
we have to add all elements in its orbit.

We recall that the standard basis of $A_4$ is expressed in terms
of the standard basis $\{\varepsilon_i\}$ of $\mathbb{R}^5$ in the
following way: $\alpha_i=\varepsilon_i-\varepsilon_{i+1}$, hence
$\alpha_5=-\alpha_1-\alpha_2-\alpha_3-\alpha_4$. The cyclic
permutation of the basis vectors of $\mathbb{R}^5$ induces the
automorphism $\gamma$ on $A_4$ ($\gamma(\alpha_i)=\alpha_{i+1}$).

\begin{proposition}\label{prop: definition of L and not -2 vectors} Let us consider the lattice $A_4(-2)^{\oplus 4}$ and the
automorphism $g:=(\gamma,\gamma,\gamma,\gamma)$ (acting as
$\gamma$ on each copy of $A_4(-2)$). Let $a_{j,i}$, $j,i=1,2,3,4$
be the element $\alpha_i$ in the $j$-copy of $A_4(-2)$. Let
\begin{equation*}\mu:= \frac{1}{2}(a_{1,1}+a_{2,1}+a_{3,1}+a_{4,1}),\ \ \
\
\nu:=\frac{1}{2}(a_{2,1}+a_{3,3}+a_{3,4}+a_{4,1}+a_{4,3}+a_{4,4}).\end{equation*}
Then: \begin{itemize} \item The lattice $$L:=A_4(-2)^{\oplus
4}+\langle g^i(\mu), g^i(\nu)\rangle_{i=0,1,2,3},$$ generated by
$A_4(-2)^{\oplus 4}$ and the 8 vectors $g^i(\mu)$, $g^i(\nu)$, for
$i=0,1,2,3$, is an even overlattice of $A_4(-2)^{\oplus 4}$ of
rank 16. \item The index of $A_4(-2)^{\oplus 4}$ in $L$ is $2^8$.
\item There are no vectors of length $-2$ in $L$.\end{itemize}
\end{proposition}
\begin{proof} Since $\mu$ has an integer intersection with the basis
$\{a_{i,j}\}$ of $A_4(-2)^{\oplus 4}$ and has self intersection
$-4$, we can add the element $\mu$ to the lattice $A_4(-2)^{\oplus
4}$ obtaining an even overlattice. All the elements $g^i(\mu)$ in
the orbit of $\mu$ have an integer intersection with this basis of
$A_4(-2)^{\oplus 4}$ and $g^i(\mu)g^i(\mu)\in \Z$ for all
$i,j=0,1,2,3$. Thus we can add the four vectors $\mu$,
$g(\mu)$, $g^2(\mu)$,
$g^3(\mu)$ to the lattice $A_4(-2)^{\oplus 4}$.\\ 
It is easy to show that the vectors $g^i(\nu)$, $i=0,1,2,3$ have
an integer intersection pairing with all the vectors in
$A_4(-2)^{\oplus 4}$, and that $g^i(\nu)g^j(\nu)\in \Z$, $(g^i(\nu))^2\in 2\Z$ (indeed these are properties of all
the vectors of type
$v_{i,j,k,h,l,m}=\frac{1}{2}(0,\varepsilon_i-\varepsilon_{i+1},
\varepsilon_{j}-\varepsilon_{j+2},
\varepsilon_k-\varepsilon_h+\varepsilon_l-\varepsilon_m)$,
$k<h<l<m$ and $\nu=v_{1,3,1,2,3,5}$).\\ Moreover
$g^i(\nu)g^j(\mu)\in \Z$, $i,j\in\{0,1,2,3\}$ (indeed this is a
property of the vectors of type $v_{i,j,k,h,l,m}$ such that
$\{i,i+1\}\cap\{j,j+2\}=\emptyset$ and
$\{k,h,l,m\}=\{i,i+1,j,j+2\}$).\\
Thus adding the vectors in the orbit of $\nu$ and of $\mu$ to
$A_4(-2)^{\oplus 4}$ we construct an even overlattice $L$ of
$A_4(-2)^{\oplus 4}$. By the computation of the discriminant of
this lattice, it follows that $A_4(-2)^{\oplus 4}$ has index $2^8$
in $L$ (indeed to construct $L$ we add to $A_4(-2)^{\oplus 4}$
exactly eight vectors of type $\frac{1}{2}v$, $v\in A_4(-2)^{\oplus 4}$ and they are independent over $\Z$).\\
Now we prove that there are no vectors with length $-2$ in $L$.
Let $y=\sum_{i=0}^3 b_ig^i(\mu)+\sum_{i=0}^3 c_ig^i(\nu)$, $b_i,
c_i\in \Z$. In $A_4(-2)^{\oplus 4}\otimes \Q$ we have:
$$\begin{array}{cl}y:=&\frac{1}{2}\left(\sum_{i=0}^3b_i\alpha_{i+1},\ \ \ \ \ \sum_{i=0}^3b_i\alpha_{i+1}+\sum_{i=0}^3c_i\alpha_{i+1},\right.\\
&\left.\sum_{i=0}^3b_i\alpha_{i+1}+(-c_1+c_3)\alpha_1+(-c_1-c_2+c_3)\alpha_2+(c_0-c_1-c_2)\alpha_3+(c_0-c_2)\alpha_4,\right.\\&\left.\sum_{i=0}^3b_i\alpha_{i+1}+(c_0-c_1+c_3)\alpha_1+(-c_2+c_3)\alpha_2+(c_0-c_1)\alpha_3+(c_0-c_2+c_3)\alpha_4\right).\end{array}$$
If we require that at least two components of $y$ are equal to
zero, we obtain that $b_i=c_i=0$ for all $i$. So if $y\neq 0$, then $y$ has at most one component equal to zero.\\
Each vector $w$ in $L$ is of the form $y+z$ with $y$ as above,
$b_i,c_i\in\{0,1\}$, and $z\in A_4(-2)^{\oplus 4}$, moreover such
$y$ and $z$ are uniquely determined by $w$.\\
If $y=(0,0,0,0)$, $w\in A_4(-2)^{\oplus 4}$ and hence $w^2\leq
-4$.\\ If $y\neq (0,0,0,0)$, then $w=\frac{1}{2}(w_1,w_2,w_3,w_4)$
with $w_i\in A_4(-2)$ and at most one of $w_i=0$. Since $w_i^2\leq
-4$ and $w_i\cdot w_j=0$ if $i\neq j$, we get $w^2\leq
\frac{3}{4}(-4)$. Hence there are no vectors of length $-2$ in
$L$. \end{proof}

\begin{proposition} The lattice $L$ is isometric to the lattice
$\Omega_{\Z/5\Z}=\Omega_{\Dh_5}$.\end{proposition} \begin{proof}
By uniqueness of $\Omega_{\Z/5\Z}$, to prove the proposition it
suffices to show that there exists a K3 surface $S$ such that
$G=\Z/5\Z$ is a group of symplectic automorphisms on $S$ and
$(H^2(S,\Z)^G)^{\perp}\simeq L$.\\
By construction $L$ admits an automorphism of order 5, $g$, acting
trivially on the discriminant group. Moreover $L$ is negative
definite and its discriminant group is $(\Z/5\Z)^4$. Hence it
admits a primitive embedding in $\Lambda_{K3}$ (\cite[Theorem
1.14.4]{Nikulin bilinear}). Since  $g$ acts trivially on the
discriminat group, $G:=\langle g\rangle$ extends to a group of
isometries on $\Lambda_{K3}$ which acts as the identity on
$L^{\perp_{\Lambda_{K3}}}$.\\
Let $S$ be a K3 surface such that $L\subset NS(S)$ (such a K3
surface exists by the surjectivity of the period map). By the
Proposition \ref{prop: definition of L and not -2 vectors}, $L$
does not contain elements of length $-2$. This is enough to prove
that the isometries of $G$ defined above (if necessary composed with a reflection in the Weil group) are effective isometries
for $S$ (the proof of this fact is essentially given in
\cite[Theorem 4.3]{Nikulin symplectic}, see also \cite[Step 4,
proof of Proposition 5.2]{symplectic prime}). By construction,
these are Hodge isometries (cf. \cite[Theorem 4.3]{Nikulin
symplectic}), so they are induced by automorphisms on $S$ (by the
Torelli theorem, cf. \cite{Burns Rapoport torelli theorem
K3}). Since these automorphisms act as the identity on $T_S\subset
L^{\perp_{\Lambda_{K3}}}$, they are symplectic.
By construction of the isometries of $G$, $L\simeq
(H^2(S,\Z)^{G})^{\perp}$, and so $L\simeq \Omega_{\Z/5\Z}$. \end{proof}

Since $\Omega_{\Dh_5} \simeq L$, on $L$ acts the dihedral group and in particular an involution. This implies that the lattice $\Omega_{\Z/2\Z}\simeq E_8(-2)$ (cf. \cite{morrison}) is a primitively embedded in $L$ and there exists an involution on $L$ acting as $-1$ on this lattice and as the identity on its orthogonal. In the following remark we give an embedding of $E_8(-2)$ in $L$ and in the proof of the Corollary \ref{cor: omegadh5 is isomteric to DIH10(16)} we describe the involution associated to this embedding.
\begin{remark}\label{rem: E8 in L}{\rm The vectors
$$\begin{array}{llll}e_1:=\mu&e_2:=g^2(\mu)+g^3(\mu)&e_3:=\nu& e_4:=\mu+g^2(\mu)+g^3(\mu)-g^2(\nu)-g^3(\nu)\\
e_5:=a_{1,1}&e_6:=a_{1,3}+a_{1,4}&e_7:=a_{2,1}&e_8:=a_{2,3}+a_{2,4}\end{array}$$
generate of a copy of $E_8(-2)$ embedded in $L$. Indeed the
lattice generated by $e_i$ is such that multiplying its bilinear
form by $\frac{1}{2}$ one obtains a negative definite even
unimodular lattice of rank 8, i.e. a copy of $E_8(-1)$.
 }\end{remark}

\begin{remark}\label{rem: 2 E8 in L} Let $f_{i+8}=g(e_i)$, $i=1,\ldots, 8$ where $\{e_i\}_{i=1,\ldots,8}$ is the basis of $E_8(-2)$ defined in the Remark \ref{rem: E8 in L}. Since $g$ is an isometry of the lattice it is clear that $f_i$, $i=9,\ldots 16$ generate a copy of $E_8(-2)$ embedded in $L$. A direct computation shows that the classes $e_i$, $f_{i+8}$ $i=1,\ldots 8$ generates a lattice of rank 16 and discriminant $5^4$ embedded in $L$ and so they are a $\Z$-basis for $L$.\end{remark}

The paper \cite{EE8- lattices and dihedral} classifies positive definite lattices which have dihedral groups $\Dh_n$ (for $n=2,3,4,5,6$) in the group of the isometries and which have the properties:
\begin{itemize}\item the lattices are rootless (i.e. there are no elements of length $2$), \item they are the sum of two copies of $E_8(2)$, \item there are two involutions in $\Dh_n$ acting as minus the identity on each copy of $E_8(2)$.\end{itemize} In \cite[Section 7]{EE8- lattices and dihedral} it is proved that there is a unique lattice with all these properties and admitting $\Dh_5$ in the group of isometries,
called $DIH_{10}(16)$.
\begin{corollary}\label{cor: omegadh5 is isomteric to DIH10(16)} The lattice $DIH_{10}(16)$ described in \cite{EE8- lattices and dihedral} is isometric to the lattice $L(-1)\simeq\Omega_{\Dh_5}(-1)$.\end{corollary}
\proof The even lattice $L$ has no vectors of length $-2$ (Proposition \ref{prop: definition of L and not -2 vectors}).\\ 
On $L$ there is an isometry of order 5, $g$ (Proposition \ref{prop: definition of L and not -2 vectors}).\\ 
Let us define a map $h$ on the lattice $L$ which acts as $-1$ on the copy of $E_8(-2)$ generated by $e_i$, $i=1,\ldots, 8$ (cf. Remark \ref{rem: E8 in L}) and as the idenitity on the orthogonal complement. Since $h$ acts trivially on the discriminant group of $E_8(-2)\simeq\langle e_i\rangle_{i=1,\ldots, 8}$, $h$ is an isometry of $L$ and in particular an involution. One can directly check that its action on the basis of $A_4(-2)^{\oplus 4}$ is $h(a_{i,1})=-a_{i,1}$, $h(a_{i,2})=-a_{i,5}$, $h(a_{i,3})=-a_{i,4}$, $h(a_{i,4})=-a_{i,3}$, $i=1,2,3,4$. This action extends to a $\Z$-basis of $L$.\\ The group $\langle g,h\rangle $ is $\Dh_5$.
The involutions $h$ and $g^2\circ h$ are two involutions generating the group $\Dh_5$. By construction $h$ and $g^2\circ h$ act as minus the identity respectively on the lattice $E_8(-2)\simeq\langle e_i\rangle_{i=1,\ldots 8}$ and on the lattice generated 
by $E_8(-2)\simeq\langle f_i\rangle_{i=9,\ldots 16}$. These two copies of $E_8(-2)$ generate $L$ (by Remark \ref{rem: 2 E8 in L}). So $L(-1)$ satisfies the conditions which define $DIH_{10}(16)$ and hence $DIH_{10}(16)\simeq L(-1)\simeq \Omega_{\Dh_5}(-1)$.

\section{Examples: algebraic K3 surfaces with a polarization of a low degree.}\label{section: examples}
Here we give some very explicit examples of families of K3 surfaces admitting $\Z/5\Z$, and hence $\Dh_5$, as group of symplectic automorphisms. In particular in this section we consider algebraic K3 surfaces.\\
We recall that a polarization $L$, with $L^2=2d$, on a K3 surface $X$ defines a map $\phi_L:X\ra \mathbb{P}^{d+1}$. In this section we consider K3 surfaces with a polarization $L$ such that $\phi_L(X)$ is a complete intersection in a certain projective space and K3 surfaces with a polarization of degree 2, which exhibits the K3 surfaces as double covers of the plane.\\ 

Let $X$ be a general member of a family of K3 surfaces admitting an automorphism of order 5, $\sigma_5$, and a polarization, $L$, invariant under $\sigma_5$. In \cite[Proposition 5.1]{symplectic prime} the possible N\'eron--Severi groups of $X$ are computed. In particular if $L^2=2d<10$, one obtains that $NS(X)\simeq \Z L\oplus \Omega_{\Z/5\Z}=:\mathcal{L}_{2d}$ . This lattice admits a unique primitive embedding in $\Lambda_{K3}$. The family of K3 surfaces with a polarization $L$, of degree $L^2< 10$, invariant under  a symplectic automorphism of order 5, is then the family of the $\mathcal{L}_{2d}$-polarized K3 surfaces. In particular for each $d<5$, we find a 3-dimensional family of K3 with such a polarization $L$ and hence we have the following possibilities:
\begin{itemize} \item $\phi_L:X\stackrel{2:1}{\ra} \mathbb{P}^2$, so $X$ is a double cover of $\mathbb{P}^2$ branched along a plane sextic curve: in this case $NS(X)\simeq \mathcal{L}_{2};$\item $\phi_L(X)$ is a quartic in $\mathbb{P}^3$: in this case $NS(X)\simeq\mathcal{L}_{4}$; \item $\phi_L(X)$
is the complete intersection of a quadric and a cubic in
$\mathbb{P}^4$: in this case $NS(X)\simeq\mathcal{L}_{6}$; \item $\phi_L(X)$ is the complete intersection of three quadrics
in $\mathbb{P}^5$: in this case $NS(X)\simeq\mathcal{L}_{8}$.
\end{itemize}
Now we construct a general member of each of these
families and show that it also admits a symplectic involution
$\iota$ generating, together with $\sigma_5$, the group $\Dh_5$.
Since the automorphisms $\iota$ and $\sigma_5$ leave invariant the
polarization, both these automorphisms can be extended to
automorphisms of the ambient projective space.\\
We will denote by $\omega$ a primitive $5$-th root of unity.
\subsection{$L^2=4$} This polarization gives a map to $\mathbb{P}^3$ where
the K3 surfaces are realized as quartic surfaces.
Let us consider the automorphism
$$\sigma_{\mathbb{P}^3}:(x_0:x_1:x_2:x_3)\ra (x_0:\omega^3 x_1:\omega x_2:\omega^2 x_3).$$
The quartic surfaces in
$\mathbb{P}^3$ defined as
\begin{equation}\label{formula: equation 5 in
P3}V(ax_0^3x_2+bx_0^2x_1^2+cx_0x_3^3+dx_0x_1x_2x_3+ex_1^3x_3+fx_1x_2^3+gx_2^2x_3^2),\end{equation}
are invariant for $\sigma_{\mathbb{P}^3}$. Hence the restriction
of $\sigma_{\mathbb{P}^3}$ to K3 surfaces with equations
\eqref{formula: equation 5 in P3} is an automorphism $\sigma_5$ of
the surfaces. To show that this automorphism is symplectic it suffices to
apply $\sigma_5$ to the following holomorphic two form in local coordinates
$x=x_1/x_0$, $y=x_2/x_0$, $z=x_3/x_0$:
$$\left(\frac{\partial f}{\partial z}\right)^{-1}dx\wedge dy,$$ where $f$
denotes the equation of the quartic in the local coordinate
$x,y,z$.\\
The equation \eqref{formula: equation 5 in P3} depends on 7
parameters. The automorphisms of $\mathbb{P}^3$ commuting with
$\sigma_{\mathbb{P}^3}$ are $diag(\alpha,\beta,\gamma,\delta)$
(which is a four dimensional group), hence this family of
$\sigma_{\mathbb{P}^3}$-invariant quartics
has $$(7-1)-(4-1)=3$$ moduli. 
So the family of K3 surfaces given by the equation \eqref{formula:
equation 5 in P3} is the family of K3 surfaces admitting an
automorphisms of order 5 leaving invariant a polarization of
degree 4.\\
Up to a projectivity, commuting with $\sigma_{\mathbb{P}^3}$, the equation \eqref{formula: equation 5 in P3} becomes
\begin{equation}\label{formula: second equation 5 in
P3}a'x_0^3x_2+b'x_0^2x_1^2+c'x_0x_3^3+d'x_0x_1x_2x_3+a'x_1^3x_3+c'x_1x_2^3+g'x_2^2x_3^2=0.\end{equation}
Let us define an involution of $\mathbb{P}^3$:
$$\iota_{\mathbb{P}^3}:(x_0:x_1:x_2:x_3)\ra (x_1:x_0:x_3:x_2).$$
The equation \eqref{formula: second equation 5 in P3} is invariant
under $\iota_{\mathbb{P}^3}$, hence $\iota_{\mathbb{P}^3}$ induces
an automorphism of the quartic surfaces with equation
\eqref{formula: second equation 5 in P3}, we call this
automorphism $\iota$. The fixed point set of
$\iota_{\mathbb{P}^3}$ in $\mathbb{P}^3$ is the union of the
lines:
$$l_1=\left\{\begin{array}{l} x_0=x_1\\ x_2=x_3\end{array}\right.,\ \ \ l_2=\left\{\begin{array}{l} x_0=-x_1\\ x_2=-x_3\end{array}\right.$$
Hence $\iota$ fixes eight points on the quartics \eqref{formula:
second equation 5 in P3}: the intersection of the quartics with
$l_1$ and $l_2$. This is enough to show that $\iota$ is a
symplectic involution, indeed the involutions which are not
symplectic either are fixed point free or fix some curves
\cite[Theorem 1]{Zhang}. Hence the quartics given in
\eqref{formula: second equation 5 in P3} (and hence, up to a
projectivity, in \eqref{formula: equation 5 in P3}) admit both the
automorphisms $\sigma_5$ and $\iota$. It is easy to check that
$\langle \sigma_{\mathbb{P}^3},\iota_{\mathbb{P}^3}\rangle
=\Dh_5$, and hence $\langle \sigma_5,\iota\rangle=\Dh_5$. So the
family of smooth quartic surfaces in $\mathbb{P}^3$ admitting the symplectic automorphism
$\sigma_5$ admits also a symplectic involution $\iota$ and in fact
the group $\Dh_5=\langle \sigma_5,\iota\rangle$.
\subsection{$L^2=6$}
This polarization gives a map to $\mathbb{P}^4$ where the K3
surfaces are realized as complete intersections of a cubic and a
quadric.\\ Let us consider the automorphism
$$\sigma_{\mathbb{P}^4}:(x_0:x_1:x_2:x_3:x_4)\ra (x_0:\omega x_1:\omega^2 x_2:\omega^3 x_3:\omega^4 x_4).$$
Let:
\begin{eqnarray}\label{formula: equation 5 in P4 quadric}\begin{array}{l} Q:=V(ax_0^2+bx_1x_4+cx_2x_3);\\
C:=V(dx_0^3+ex_0x_1x_4+fx_0x_2x_3+gx_1^2x_3+hx_2x_4^2+lx_1x_2^2+mx_3^2x_4),\end{array}\end{eqnarray}
then $Q$ and $C$ are $\sigma_{\mathbb{P}^4}$-invariant
hypersurfaces in $\mathbb{P}^4$. We observe that the complete
intersection of these two hypersurfaces is generically smooth
and thus it is a K3 surface.\\
The complete intersection of $Q$ and $C$ is also the
complete intersection of $Q$ and $C+ \lambda x_0Q$. Hence there is
a 1-dimensional family of invariant cubics giving the same
complete intersection: the cubics giving different complete
intersections depend on $7-1=6$ parameters. The automorphisms of
$\mathbb{P}^4$ which commute with $\sigma_{\mathbb{P}^4}$ are of
the form $diag(\alpha, \beta, \gamma, \delta, \varepsilon)$. So
the family of complete intersections of a cubic and a quadric 
invariant under the automorphism $\sigma_{\mathbb{P}^4}$ has
$(3-1)+(6-1)-(5-1)=3$
moduli.\\
Let $X$ be the complete intersection of $Q$ and $C$. 
The automorphism $\sigma_{\mathbb{P}^4}$ induces a symplectic automorphism
on $X$ (this can be shown as in case $L^2=4$
considering the two holomorphic form, in local coordinates
$x,y,z,t$, $(dx\wedge dy)/(Q_zC_t-C_tQ_z)$ where $F_x$ is the
partial derivative of $F$
w.r.t.\ $x$).\\
Up to the action of the projectivities commuting with
$\sigma_{\mathbb{P}^4}$, we can assume that $g=h$ and $l=m$ in the
equation of $C$, in \eqref{formula: equation 5 in P4 quadric}.
Hence the involution
$$\iota_{\mathbb{P}^4}:(x_0:x_1:x_2:x_3:x_4)\ra (x_0:x_4:x_3:x_2:x_1)$$
fixes $Q$ and $C$. So its restriction to
$X$ is an involution, $\iota$, of $X$. Moreover $\iota$ has eight
fixed points (six on the plane $x_1=x_4, x_2=x_3$ and two on the
line $x_0=0$, $x_1+x_4=0$, $x_2+x_3=0$). Thus $\sigma_5$ and
$\iota$ are symplectic automorphisms of the K3 surface $X$ and
they generate the group $\Dh_5$.

\subsection{$L^2=8$}
This polarization gives a map to $\mathbb{P}^5$ where the K3
surfaces are realized as complete intersections three
quadrics.\\
Let us consider the map
$$\sigma_{\mathbb{P}^5}:(x_0:x_1:x_2:x_3:x_4:x_5)\ra (x_0:x_1:\omega x_2:\omega^2 x_3:\omega^3 x_4:\omega^4 x_5)$$
and the complete intersection of the quadrics
\begin{eqnarray*}\label{formula: equation 5 in P5
quadric}
\begin{array}{lll} Q_1':=&V(ax_0^2+bx_0x_1+cx_1^2+dx_2x_5+ex_3x_4),\\
Q_2':=&V(fx_1x_2+gx_3x_5+hx_4^2),\\
Q_3':=&V(lx_1x_5+mx_2x_4+nx_3^2).
\end{array}\end{eqnarray*}
The group of automorphisms of $\mathbb{P}^5$ commuting with
$\sigma_{\mathbb{P}^5}$ is $(GL(2)\times GL(1)^4)/GL(1)$ which has
dimension $7=8-1$. So these complete intersections in
$\mathbb{P}^5$ have $(5-1)+(4-1)+(4-1)-(8-1)=3$ moduli. Up to
automorphisms of $\mathbb{P}^5$ commuting with
$\sigma_{\mathbb{P}^5}$ we can assume that the quadrics have the
following equation:
\begin{eqnarray*}
\begin{array}{lll} Q_1:=&V(x_0^2+bx_0x_1+x_1^2+dx_2x_5+ex_3x_4),\\
Q_2:=&V(x_1x_2+x_3x_5+x_4^2),\\
Q_3:=&V(x_1x_5+x_2x_4+x_3^2),
\end{array}\end{eqnarray*}
and in fact they depend on 3 parameters.\\
The complete intersection $X$ of the quadrics $Q_1$, $Q_2$, $Q_3$
is smooth for a generic choice of the parameters $b,d,e$ (one can
check it directly putting $e=b=0$, $d=1$). Moreover $X$ is
invariant under the automorphism $\sigma_{\mathbb{P}^5}$, so
$\sigma_{\mathbb{P}^5}$ induces an automorphism $\sigma_5$ on $X$
and $\sigma_5$ is symplectic (this can be shown as in the case $L^2=6$).\\
The involution
$$\iota_{\mathbb{P}^5}:(x_0:x_1:x_2:x_3:x_4:x_5)\ra (x_0:x_1:x_5:x_4:x_3:x_2)$$
fixes the quadric $Q_1$ and switches the quadrics $Q_2$ and $Q_3$.
So its restriction to the K3 surface $X$ is an involution,
$\iota$, of the surface $X$. Moreover $\iota$ has eight fixed
points, on the space $x_2=x_5$, $x_3=x_4$. Thus $\sigma_5$ and
$\iota$ are symplectic automorphisms of the K3 surface $X$ and
they generate the group $\Dh_5$.

\subsection{$L^2=2$}
This polarization gives a $2:1$ map to $\mathbb{P}^2$ and the K3
surfaces are realized as double cover of $\mathbb{P}^2$ branched along a sextic plane curve.\\
The map
$$\sigma_{\mathbb{P}^2}:(x_0:x_1:x_2)\ra (x_0:\omega x_1:\omega^4 x_2)$$
is an automorphism of $\mathbb{P}^2$. Up to projectivity of $\mathbb{P}^2$ commuting with $\sigma_{\mathbb{P}^2}$, the invariant sextic for $\sigma_{\mathbb{P}^2}$ are 
\begin{eqnarray*}\mathcal{C}_6:= V(x_0^6+x_0x_1^5+x_0x_2^5+ax_0^4x_1^2x_2^2+bx_0^3x_1^3x_2^3+cx_1^3x_2^3)
\end{eqnarray*}
Let $X$ be the double cover of $\mathbb{P}^2$ branched along $\mathcal{C}_6$, i.e. $X$ is $V(u^2-(x_0^6+x_0x_1^5+x_0x_2^5+ax_0^4x_1^2x_2^2+bx_0^3x_1^3x_2^3+cx_1^3x_2^3))$ in the weighted projective space $\mathbb{PW}(3,1,1,1)$.
The automorphism $\sigma_{\mathbb{P}^2}$ lifts to a symplectic automorphism $\sigma_5:(u:x_0:x_1:x_2)\ra (u:x_0:\omega x_1:\omega^4 x_2)$ of  $X$. So we constructed the 3-dimensional family of K3 surfaces which are double covers of $\mathbb{P}^2$ and have a symplectic automorphism of order 5 which leavs invarinat the polarization.\\
The involution $$\alpha_{\mathbb{P}^2}:(x_0:x_1:x_2)\ra (x_0:x_2:x_1)$$ leaves the curve $\mathcal{C}_6$ invariant, so it lifts to an involution $\alpha_X:(u:x_0:x_1:x_2)\ra (u:x_0:x_2:x_1)$ of the surface $X$. The involution $\alpha_X$ fixes a curve (the pull back of the line $x_1=x_2$ in $\mathbb{P}^2$), so it is not symplectic. Let $i:(u:x_0:x_1:x_2)\ra (-u:x_0:x_1:x_2)$ be the covering involution on $X$. It is a non symplectic involution (indeed the quotient $X/\iota$ is rational) and it commutes both with $\alpha_X$ and $\sigma_5$. The involution $\iota=\alpha_X\circ i$ is a symplectic involution on $X$ (because it is the composition of two commuting non symplectic involutions). Moreover one has $\iota\circ \sigma_5=\sigma_5^{-1}\circ\iota$, and hence $\Dh_5=\langle \sigma_5, \iota\rangle$ acts symplectically on $X$.

\bibliographystyle{amsplain}

\end{document}